\documentclass[11pt]{article}
\usepackage{amsmath}
\usepackage{amsthm}
\usepackage{amssymb,amsfonts}
\usepackage[T1]{fontenc}
\usepackage{hyperref}
\usepackage{latexsym}
\usepackage{todonotes}
\usepackage{nicefrac}
\usepackage{epsfig}
\usepackage{stmaryrd}
\usepackage{setspace}
\usepackage{enumerate}
\usepackage[all]{xypic}
\usepackage{bbm,ifpdf,tikz}
\ifpdf
\usepackage{pdfsync}
\fi

\newtheorem{theorem}{Theorem}[section]

\newtheorem{lemma}[theorem]{Lemma}
\newtheorem{proposition}[theorem]{Proposition}
\newtheorem{corollary}[theorem]{Corollary}

{
\theoremstyle{definition}

\newtheorem{example}[theorem]{Example}

\newtheorem{remark}[theorem]{Remark}
}

\newcommand{\excise}[1]{}
\newcommand{\Spec}{\operatorname{Spec}}

\newcommand{\cone}{{\operatorname{cone}}}

\renewcommand{\dim}{\operatorname{dim}}

\newcommand{\crk}{\operatorname{crk}}

\renewcommand{\and}{\qquad\text{and}\qquad}

\newcommand{\Hom}{\operatorname{Hom}}

\newcommand{\Z}{\mathbb{Z}}
\newcommand{\Q}{\mathbb{Q}}
\newcommand{\N}{\mathbb{N}}

\newcommand{\C}{\mathbb{C}}

\newcommand{\IH}{I\! H}

\newcommand{\cC}{\mathcal{C}}

\newcommand{\OS}{OS}

\newcommand{\Rep}{\operatorname{Rep}}

\newcommand{\op}{{\operatorname{op}}}
\newcommand{\Comp}{\operatorname{Comp}}

\newcommand{\cGop}{\cG^\op}
\newcommand{\cTop}{\cT^\op}
\newcommand{\cG}{\mathcal{G}}
\newcommand{\cT}{\mathcal{T}}
\newcommand{\cRT}{\mathcal{RT}}
\newcommand{\RTi}{\mathbf{T}}
\newcommand{\cRTop}{\cRT^{\op}}
\newcommand{\OI}{\operatorname{OI}}
\newcommand{\ue}{\underline{e}}
\newcommand{\uf}{\underline{f}}
\newcommand{\um}{\underline{m}}
\newcommand{\un}{\underline{n}}
\newcommand{\uv}{\underline{v}}
\newcommand{\OIr}{\OI^{\times r}}

\newcommand{\bigmid}{\;\big{|}\;}
\newcommand{\Bigmid}{\;\Big{|}\;}
\newcommand{\Edge}{\operatorname{Edge}}
\renewcommand{\Vert}{\operatorname{Vert}}
\newcommand{\UConf}{\operatorname{UConf}}
\newcommand{\Sw}{\widetilde{S}}

\newcommand{\Aut}{{\operatorname{Aut}}}
\newcommand{\tS}{\widetilde{S}}
\newcommand{\ith}{i^{\operatorname{th}}}
\newcommand{\tith}{2i^{\operatorname{th}}}

\begin{document}
\spacing{1.2}
\noindent{\LARGE\bf Functorial invariants of trees and their cones}\\

\noindent{\bf Nicholas Proudfoot and Eric Ramos}\\
Department of Mathematics, University of Oregon,
Eugene, OR 97403\\

{\small
\begin{quote}
\noindent {\em Abstract.}
We study the category whose objects are trees (with or without roots) and whose morphisms
are contractions.  We show that the corresponding contravariant module categories are locally Noetherian,
and we study two natural families of modules over these categories.  The first takes a tree to a graded piece of the homology
of its unordered configuration space, or to the homology of the unordered configuration space of its cone.
The second takes a tree to a graded piece of the intersection homology of the reciprocal plane of its cone, which is a vector space
whose dimension is given by a Kazhdan-Lusztig coefficient.
We prove finite generation results for each of these modules, which allow us to draw conclusions about the growth
of Betti numbers of configuration spaces and of Kazhdan-Lusztig coefficients of graphical matroids.
\end{quote} }

\section{Introduction}
Our aim in this paper is to study two different ways to assign an abelian group or a vector space to a graph, both of which 
are contravariantly functorial with respect to contractions.  The first assigns to a graph $G$ the $\ith$ homology group of the 
unordered configuration space of $n$ points on $G$.  It is not obvious that a contraction of graphs induces a map
on homology groups of configuration spaces; this follows from the fact that this homology group can be computed using the 
reduced \'Swi\k{a}tkowski complex \cite{Sw-graphs, ADK}, which is itself functorial with respect to contractions.
The second assigns to a graph $G$ the $\tith$ intersection homology group of a certain algebraic variety $X_G$ called the reciprocal plane of $G$.
A contraction of graphs induces an inclusion of reciprocal planes, 
which in turn induces a map on intersection cohomology.\footnote{The induced
map on varieties is not quite canonical, but it is canonical enough that the induced map on intersection homology does not depend on any choices.}
% The dimension of the $\tith$ intersection homology of $X_G$ is equal to the coefficient of $t^i$ in the Kazhdan-Lusztig polynomial of 
% the matroid associated with $G$ \cite{EPW}, so we will refer to this vector space invariant as the $\ith$ Kazhdan-Lusztig coefficient of $G$.

Let $\cG$ be the category whose objects are graphs and whose morphisms are contractions (see Section \ref{sec:trees} for
a more precise definition).
Both of the aforementioned procedures of assigning an abelian group or vector space to a graph may be regarded as
functors from the opposite category $\cGop$ to the category of finite dimensional modules over a Noetherian commutative ring $k$.
(For the homology of configuration spaces we will usually want to take $k$ to be $\Z$ or $\Q$, whereas for the intersection homology
of the reciprocal plane we will take $k$ to be $\C$.)  Such functors are called $\cGop$-modules, and the category that they form is called $\Rep_k(\cGop)$.
Unfortunately, this category is not well behaved.  In particular, it is not locally Noetherian: there is a natural notion
of finite generation for a $\cGop$-module, and a submodule of a finitely generated module need not itself be finitely generated.
Since both of our modules are computed by passing to the homology of some complex of modules, this means that we have no hope of proving any
finite generation results.

We deal with this difficulty by working with the subcategory $\cT\subset\cG$ consisting of trees, which has much nicer properties.
Homology groups of configuration spaces of trees are already relatively well understood \cite{Farley, MaSa, Ramos-Graph}, 
and Kazhdan-Lusztig coefficients of trees are trivial in positive degree.  However, we can obtain much more interesting results
by looking at the cone over a tree, which is obtained by adding a single new vertex and connecting it to all other vertices.
For example, the cone over a path is a fan (see the figure in Example \ref{cone-flat}), and the cone over 
the complete bipartite graph $K_{m,1}$ is the thagomizer graph \cite{thag}.  A contraction of trees
more or less
% \footnote{There is a technical difficulty here owing to the fact that the contraction 
% of a cone is isomorphic to the cone of a contraction with some repeated edges.  For Kazhdan-Lusztig coefficients this does not pose any problems,
% but for homology of configuration spaces it forces us to work with the category of rooted trees, which also has a Noetherian module category.}
induces a contraction of cones of trees (see Remark \ref{more or less}), and we therefore obtain modules over our tree category.

\begin{remark}
This operation of taking the cone over a graph is very natural from the point of view of matroid theory, since the graph may be recovered
from the matroid associated with its cone (two vertices are connected by an edge if and only if the corresponding three
edges of the cone form a cycle), but not from the matroid associated with the graph itself.
\end{remark}

\subsection{Categorical results}
We now state our main results about tree categories, each of which is proved in Section \ref{sec:tree-cats}.
We have already introduced the category $\cT$ of trees whose morphisms are contractions.  Let $\cRT$ be the category of rooted
trees, which means that we mark one vertex and require our contractions to take the root to the root.
For any integer $l\geq 2$, let $\cT_l\subset \cT$ be the full subcategory consisting of trees with at most $l$ leaves.
The following result is proved in Section \ref{sec:trees}.

\begin{theorem}\label{Noetherian}
Fix a commutative Noetherian ring $k$.  The categories $\Rep_k(\cTop)$, $\Rep_k(\cRTop)$, and $\Rep_k(\cT_l^\op)$
are all locally Noetherian.
\end{theorem}

One of our main motivations is to study the growth of the dimensions of various modules in these categories.
There is a notion of degree of generation, and we call a module $d$-small if it
is isomorphic to a subquotient of a module that is generated in degree at most $d$, and $d$-smallish
if it admits a filtration whose associated graded is $d$-small.  Theorem \ref{Noetherian} easily implies that $d$-smallish
modules are themselves finitely generated (Proposition \ref{smallish fg}), though not necessarily in degrees $\leq d$.
The next result is proved in Section \ref{sec:growth}.

\begin{theorem}\label{bounded growth}
Suppose that $k$ is a field and $M$ is a $d$-smallish $k$-linear module over $\cTop$, $\cRTop$, or $\cT_l^\op$.
There exists a polynomial $f_M(t)$ with the property that, for any tree $T$, $\dim_k M(T) \leq f_M(|T|)$, where $|T|$
is the number of edges of $T$.
\end{theorem}

It is impossible to ask $\dim_k M(T)$ to be equal to a polynomial in $|T|$, because this dimension typically depends
on more than just the number of edges.  However, there are certain operations that we can perform on a tree that cause
the dimension of a smallish module to grow polynomially.  We will state the next result in a very informal way, and ask the
reader to consult Sections \ref{sec:subdivision}-\ref{sec:growth} for more precise formulations.  The relevant theorems are
Theorems \ref{actual polynomial} and \ref{actual polynomial 2}.

\begin{theorem}
Suppose that $k$ is a field and $M$ is a $d$-smallish object of $\Rep_k(\cTop)$ or $\Rep_k(\cRTop)$.
Fix a tree $T$, and build new trees from $T$ via either 
via subdivision (breaking finitely many edges up into paths) or sprouting (adding new leaves at finitely many vertices).
The dimension of $M$ evaluated at one of these new trees is eventually equal to a polynomial of degree at most $d$ in the new parameters.
If $M$ is a $d$-smallish object of $\Rep_k(\cT_l^\op)$, the same result holds for subdivision.
\end{theorem}

\subsection{Homology of configuration spaces}
Graph configuration spaces have been extensively studied in settings both theoretical \cite{ADK,A,KP} and applied \cite{F}.
The idea of fixing the number of points and varying the underlying graph has been explored in a number of recent works \cite{RW,R,L}.
We focus on trees and their cones, and obtain the following results.

\begin{theorem}\label{treefing}
Fix natural numbers $n$ and $i$. The $\cTop$-module
\[
T \mapsto H_i\big(\UConf_n(T); \Z\big)
\]
is $(n+i)$-small.  In particular, it is finitely generated.
\end{theorem}

\begin{theorem}\label{conetreefing}
Fix natural numbers $n$ and $i$. The $\cRTop$-module
\[
(T,v) \mapsto H_i\big(\UConf_n(\cone(T)); \Z\big)
\]
is $(n+i)$-small.  In particular, it is finitely generated.
\end{theorem}

\begin{remark}\label{more or less}
It may seem funny that the second module in Theorem \ref{conetreefing} is a module over $\cRTop$ rather than $\cTop$,
since the configuration space itself is not sensitive to the choice of root.  The issue is that a contraction of trees
does not quite induce a contraction of cones, and we use the choice of root to define the maps in a natural way.
This fix is not needed in the setting of the next application (see Remark \ref{simplification}), so we will again be able to work with unrooted trees.
\end{remark}

\begin{remark}
In this work we only consider \textit{unordered} configurations of points. This is done mainly because the tools we use largely derive from the paper \cite{ADK}, 
and this is the setting in which they work. It seems very likely that one can obtain analogues of Theorems \ref{treefing} and \ref{conetreefing} 
for \textit{ordered} configuration spaces. To accomplish this, the first step would be to reprove certain technical results from \cite{ADK} in the setting of ordered configuration spaces. 
\end{remark}

\subsection{Kazhdan-Lusztig coefficients}
Kazhdan-Lusztig polynomials of matroids were introduced in \cite{EPW}, and are in many ways
analogous to the Kazhdan-Lusztig polynomials that appear in Lie theory \cite[Section 2.5]{KLS}.
If a matroid comes from a graph (or more generally from a hyperplane arrangement),
the coefficient of $t^i$ in its Kazhdan-Lusztig polynomial is equal to the dimension of the $\tith$ intersection homology
group of the reciprocal plane $X_G$ \cite{EPW}.  

Kazhdan-Lusztig coefficients of graphical matroids have been the subject of many recent papers \cite{PWY,thag,fs-braid,fan-wheel-whirl},
with only a small number of special families being explicitly understood.
An interesting special case is the thagomizer graph, which is the cone over the tree $K_{m,1}$.
The Kazhdan-Lusztig coefficients of this graph grow faster than any polynomial \cite{thag}, so Theorem \ref{bounded growth}
tells us that the corresponding module over $\cTop$ cannot be finitely generated.  The problem goes away, however, if we restrict our attention
to trees with a bounded number of leaves.  The following theorem is proved in Section \ref{sec:kl-thm}.

\begin{theorem}\label{IH-small}
Fix natural numbers $l$ and $i$.
The $\C$-linear $\cT_l^\op$-module
\[
T \mapsto \IH_{2i}\big(X_{\cone(T)}\big)
\]
is $(2i+l-2)$-smallish.  In particular, it is finitely generated.
\end{theorem}

\subsection{Future work}
In a future paper, we will prove analogous results for the full subcategory of $\cG$ consisting of connected graphs with fixed Euler characteristic.  
Despite the fact that trees form a special case (Euler characteristic equal to 1), 
it is a case made richer by the notion of rooted trees, which allows us to work with cones (see Remark \ref{more or less}).

\vspace{\baselineskip}
\noindent
{\em Acknowledgments:}
NP is supported by NSF grant DMS-1565036.  ER is supported
by NSF grant DMS-1704811. The second author would also like to send thanks to Daniel L\"utgehetmann for various discussions related to Theorem \ref{treefing} prior to this work.

\section{Tree categories}\label{sec:tree-cats}
The goal of this section is to give precise definitions of the tree categories $\cT$, $\cRT$, and $\cT_l$;
to prove that their contravariant module categories are locally Noetherian; and to study the dimension growth
of finitely generated modules.

\subsection{Gr\"obner categories and Noetherianity}
Fix a Noetherian commutative ring $k$.  Given an essentially small category $\cC$, we will be interested in the Abelian
category $\Rep_k(\cC)$ of covariant functors from $\cC$ to the category of $k$-modules.
Such a functor will be called a {\bf \boldmath{$\cC$}-module}.
If $x$ is an object of $\cC$, we define the {\bf principal projective} $\cC$-module $P_x\in\Rep_k(\cC)$
by letting $P_x(y)$ be the free $k$-module with basis $\Hom_\cC(x,y)$ and defining morphisms by composition.  
An arbitrary $\cC$-module $M$ is called {\bf finitely generated} if it admits a surjection from a direct sum
of finitely many principal projectives.  The category $\Rep_k(\cC)$ is called {\bf locally Noetherian} if every submodule
of a finitely generated $\cC$-module is itself finitely generated.

Sam and Snowden introduce the notions of Gr\"obner categories and quasi-Gr\"obner categories
as a means of proving that the corresponding module categories are locally Noetherian \cite{sam}.
Let $\cC$ be an essentially small category.  For any object $x$ of $\cC$, let $\cC_x$ denote the set of all morphisms in $\cC$
with domain $x$.  This set is equipped with a partial order by putting 
$$f \leq g \iff g = h \circ f \text{ for some morphism $h$.}$$
The category $\cC$ is said to have \textbf{property (G2)} if, for all objects $x$, the poset $\cC_x$ is Noetherian,
which means that all descending chains stabilize and there are no infinite anti-chains.
The category $\cC$ is said to have \textbf{property (G1)} if, for all objects $x$, there exists some linear order 
$\preccurlyeq$ on $\cC_x$ such that, for all monomials $f,g \in \cC_x$ with the same target $y$
and all morphisms $h\in\cC_y$,
$$f \preccurlyeq g \implies h\circ f \preccurlyeq h \circ g.$$
The category $\cC$ is \textbf{Gr\"obner} if it has properties (G1) and (G2) and objects of $\cC$
have no nontrivial endomorphisms.

Let $\cC$ and $\cC'$ be essentially small categories.  A functor $\Phi: \cC \rightarrow \cC'$ is said to have \textbf{property (F)}
if, given any object $x$ of $\cC'$, there exist finitely many objects $y_1, \ldots, y_n$ of $\cC$
and morphisms $f_i : x \rightarrow \Phi(y_i)$ in $\cC'$ such that for any object $y$ of $\cC$ and any morphism 
$f : x \rightarrow \Phi(y)$ in $\cC'$, there exists a morphism $g : y_i \rightarrow y$ in $\cC$ such that $f = \Phi(g) \circ f_i$.
An essentially small category $\cC'$ is \textbf{quasi-Gr\"obner} if there exists a Gr\"obner category $\cC$ and an essentially surjective
functor $\Phi: \cC \rightarrow \cC'$ with property (F).

\begin{remark}\label{prop-F}
It is easy to see that property (F) is closed under composition \cite[Proposition 3.2.6]{sam}.
Thus, if $\cC$ is quasi-Gr\"obner and $\Phi: \cC \rightarrow \cC'$ has property (F), then $\cC'$
is also quasi-Gr\"obner.
\end{remark}

The motivation for these definitions comes from the following two theorems, both of which of fundamental
importance in our work.

\begin{theorem}\label{fg}{\em \cite[Proposition 3.2.3]{sam}}
If $\Phi: \cC \rightarrow \cC'$ has property (F) and $M$ is a finitely generated $\cC'$-module,
then $\Phi^*M$ is a finitely generated $\cC$-module.
\end{theorem}

\begin{theorem}\label{qgn}{\em \cite[Theorem 1.1.3]{sam}}
If $\cC$ is quasi-Gr\"obner, then the module category $\Rep_k(\cC)$ is locally Noetherian.
\end{theorem}

\subsection{Trees, rooted and otherwise}\label{sec:trees}
A {\bf graph} is a finite CW complex of dimension at most 1.  
The 0-cells are called {\bf vertices} and the 1-cells are called {\bf edges}.
If $f:G\to G'$ is a map of CW complexes, we say that $f$ is {\bf very cellular} if it takes every vertex to a vertex
and every edge to either a vertex or an edge.  An edge that maps to a vertex will be called a {\bf contracted edge}.
If $G$ and $G'$ are graphs, we define a {\bf graph morphism} from $G$ to $G'$ to be an equivalence class of very cellular maps,
where two very cellular maps are equivalent if and only if they are homotopic through very cellular maps.\footnote{If we are content
to work with graphs without loops, then we could equivalently define a graph to be a simplicial complex of dimension at most 1
and a graph morphism to be a simplicial map.  We will not care about loops in this paper, but they will be necessary in the sequel paper,
so we are giving this more general definition here.}
We note that a graph morphism $\varphi:G\to G'$ induces a well defined map on vertex sets, and it also makes sense to talk
about the set of edges that are contracted by $\varphi$.
We say that a graph morphism is a {\bf contraction} if it may be represented by a very cellular map that is
a surjective homotopy equivalence with connected fibers.

A {\bf tree} is a graph that admits a contraction to a point.
If $T$ and $T'$ are trees, a contraction from $T$ to $T'$ is uniquely determined by the induced map on vertices,
which can be any map with the property that the preimage of any vertex in $T'$ is equal to the set of vertices of a subtree of $T$.
A {\bf rooted tree} is a tree along with a choice of vertex.
Let $\cT$ be the category of trees with contractions, and let $\cRT$ be the category of rooted trees with contractions that preserve the root.
We have a forgetful functor $\Phi:\cRT\to\cT$.

Any rooted tree has a natural partial order on its vertex set, where the root is maximal,
and more generally $v\leq w$ if and only if the unique path from $v$ to the root passes through $w$.
Barter \cite{Barter} studies the category $\RTi$ whose objects are rooted trees and whose morphisms
are pointed order embeddings of vertex sets.

\begin{proposition}\label{opposite}
The category $\cRTop$ is equivalent to $\RTi$.
\end{proposition}

\begin{proof}
Let $(T,v)$ and $(T',v')$ be rooted trees.  
Given a contraction $\varphi:(T,v)\to (T',v')$ in $\cRT$, we construct a morphism $\varphi^*:(T',v')\to (T,v)$ in $\RTi$
by sending each vertex of $T'$ to the maximal vertex in its preimage.  Conversely,
given a morphism $\psi:(T',v')\to (T,v)$ in $\RTi$, we construct a contraction
$\psi^*:(T,v)\to (T',v')$ in $\cRT$ that sends each vertex $w$ of $T$ to the minimal vertex of $T'$ whose image under $\psi$ lies weakly above $w$.
It is easy to see that $\varphi^{**} = \varphi$ and $\psi^{**} = \psi$, thus these two constructions are mutually inverse.
\end{proof}

\begin{example}\label{Barter}
The following illustration depicts a morphism in the category $\cRT$ alongside the corresponding morphism in the category $\RTi$.
The fat vertices represent the roots.
\begin{figure}[ht]
\begin{center}
\includegraphics[width=5cm]{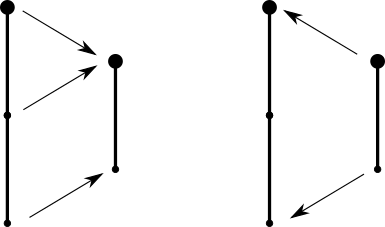}
\end{center}
% \caption{A map from $(T,v)$ to $(T',v')$ in the category $\cRT$ and the corresponding map from $(T',v')$ to $(T,v)$ in the category $\RTi$.}
\end{figure}
\end{example}

\begin{corollary}\label{cRTop-qg}
The categories $\cRTop$ and $\cTop$ are both quasi-Gr\"obner.
\end{corollary}

\begin{proof}
Barter proves that $\RTi$ is quasi-Gr\"obner \cite[Theorem 5]{Barter}, thus so is $\cRTop$.
The forgetful functor $\Phi^\op:\cRTop\to\cTop$ is surjective on both objects and morphisms, 
and therefore has property (F).  It follows from Remark \ref{prop-F} that $\cTop$ is also quasi-Gr\"obner.
\end{proof}

\begin{proof}[Proof of Theorem \ref{Noetherian}]
The fact that the categories $\Rep_k(\cRTop)$ and $\Rep_k(\cTop)$ are locally Noetherian
follows from Theorem \ref{qgn} Corollary \ref{cRTop-qg}.
The category $\Rep_k(\cT_l^\op)$ can be identified with the full subcategory of $\Rep_k(\cTop)$
consisting of modules that evaluate to zero on any tree with more than $l$ leaves, and a module 
in $\Rep_k(\cT_l^\op)$ is finitely generated over $\cT_l^\op$ if and only if it is finitely generated over $\cTop$.
Thus the local Noetherian property for $\Rep_k(\cT_l^\op)$ follows from the local Noetherian property for $\Rep_k(\cTop)$.
\end{proof}

\begin{remark}
One can also prove that $\Rep_k(\cT_l^\op)$ is locally Noetherian by showing the category $\cT_l^\op$ is quasi-Gr\"obner;
this would involve mimicking Barter's argument in the setting of trees with at most $l$ leaves.
\end{remark}

\subsection{Subdivision}\label{sec:subdivision}
Fix a tree $T$, a natural number $r$, and an ordered $r$-tuple $\ue = (e_1,\ldots,e_r)$ of distinct directed
edges of $T$.  For any ordered $r$-tuple $\um = (m_1,\ldots,m_r)$ of natural numbers, let $T(\ue,\um)$ be 
the tree obtained from $T$ by subdividing each edge $e_i$ into $m_i$ edges.  
The number $m_i$ is allowed to be zero, and we adopt the convention
that subdividing $e_i$ into 0 edges means contracting $e_i$.  
For each $i$, the tree $T(\ue,\um)$ has a directed path of length $m_i$ where the directed edge $e_i$ used to be, 
and we label the vertices of that path $v_i^0,\ldots,v_i^{m_i}$.

Let $\OI$ be the category whose objects are linearly ordered finite sets and whose morphisms are ordered inclusions.
Every object of $\OI$ is isomorphic via a unique isomorphism to the finite set $[m]$ for some $m\in\N$.
For any $\um\in\N^r$, let $[\um]$ denote the corresponding object of the product category $\OIr$.

Our goal in this section is to define a {\bf subdivision functor} $\Phi_{T,\ue}:\OIr\to\cTop$ and prove that $\Phi_{T,\ue}$ 
has property (F).  
We define our functor on objects by putting $\Phi_{T,\ue}([\um]) := T(\ue,\um)$.
Let $\uf = (f_1,\ldots,f_r)$ be a morphism in $\OIr$ from $[\um]$ to $[\un]$.  
We define the corresponding contraction $$\Phi_{T,\ue}(\uf):T(\ue,\un)\to T(\ue,\um)$$
by sending $v_i^t$ to $v_i^s$, where $s$ is the maximal element of the set $\{0\}\cup \{j\mid f_i(j)\leq t\} \subset \{0,1,\ldots,m_i\}$.

\begin{example}
If $T$ consists of a single edge and $r=1$, then the essential image of $\Phi_{T,\ue}$ 
is equal to the path category $\cT_2^\op$.
On the other hand, $\OI$ may be identified with the full subcategory of $\RTi$ consisting of rooted paths with the root at an endpoint,
where the ordered set $[m]$ goes to the standard path on vertex set $\{0,\ldots,m\}$ with root 0.
(It is a regretable convention that the root is the maximal element of the vertex set of a rooted tree,
so this identification reverses the order on $[m]$.)
The functor $\Phi_{T,\ue}$ can be identified with the composition of the equivalence from Proposition \ref{opposite}
(restricted to paths) with the functor that forgets the root.
\end{example}

For any $\un \in\N^r$, let $|\un| := \sum n_i$.
Recall that for any tree $R$, we have defined $|R|$ to be the number of edges of $R$.
We say that a contraction $\varphi:T(\ue,\un)\to R$ 
{\bf factors nontrivially} if there exists a non-identity morphism $\uf:[\um]\to[\un]$ in $\OIr$
and a contraction $\psi:T(\ue,\um)\to R$ such that $\varphi = \psi\circ\Phi_{T,\ue}(\uf)$.

\begin{proposition}\label{OIrF}
The subdivision functor $\Phi_{T,\ue}:\OIr\to\cTop$ has property (F).
\end{proposition}

\begin{proof}
Property (F) says exactly that, for any tree $R$, the set of contractions 
from some $T(\ue,\um)$ to $R$ that do not factor nontrivially is finite.
Let $\varphi:T(\ue,\um)\to R$ be given.
We have $$|T(\ue,\um)| = |T| + |\um| - r,$$ so $\varphi$ must contract $|T| + |\um| - r - |R|$ edges.
If $|\um|$ is sufficiently large, then at least one of those edges must be one of the subdivided edges.
We may then factor $\varphi$ nontrivially by first contracting that edge.

This tells us that, if we are looking for contractions from some $T(\ue,\um)$ to $R$ that do not factor nontrivially,
we only need to consider finitely many $r$-tuples $\um$.  The proposition then
follows from the fact that all Hom sets in $\cT$ are finite.
\end{proof}

\begin{remark}\label{other-cats-F}
It will be convenient to record a few variants of Proposition \ref{OIrF} in which the category $\cTop$
is replaced by other closely related tree categories.  For example, if $T$ is rooted, then we get a functor
from $\OIr$ to $\cRTop$, which also has property (F).  If $T$ has at most $l$ leaves, then we get
a functor from $\OIr$ to $\cT_l^\op$, and this functor also has property (F).  
Both of these statements are proved in exactly the same way as Proposition \ref{OIrF}.
\end{remark}

\subsection{Sprouting}

Fix a tree $T$, a natural number $r$, and an ordered $r$-tuple $\uv := (v_1,\ldots,v_r)$ of distinct vertices of $T$.
For any ordered $r$-tuple $\um = (m_1,\ldots,m_r)$ of natural numbers, let $T(\uv,\um)$ be 
the tree obtained from $T$ by attaching $m_i$ new edges to the vertex $v_i$, each of which has a new leaf as its other endpoint.
We will label the new leaves connected to the vertex $v_i$ by the symbols $v_i^1,\ldots,v_i^{m_i}$.

Our goal in this section is to define a {\bf sprouting functor} $\Psi_{T,\uv}:\OIr\to\cTop$ and prove that $\Psi_{T,\uv}$ 
has property (F).  
We define our functor on objects by putting $\Psi_{T,\ue}([\um]) := T(\uv,\um)$.
Let $\uf = (f_1,\ldots,f_r)$ be a morphism in $\OIr$ from $[\um]$ to $[\un]$.  
We define the corresponding contraction $$\Psi_{T,\uv}(\uf):T(\uv,\un)\to T(\uv,\um)$$
by fixing all of the vertices of $T$, sending $v_i^t$ to $v_i^s$ if $f_i(s) = t$, and sending $v_i^t$ to $v_i$ of $t$ is not in the image
of $f_i$.

\begin{example}
If $T$ consists of a single vertex and $r=1$, then the essential image of $\Psi_{T,\uv}$ 
is equal to the category consisting of the graphs $K_{m,1}$ with one central vertex connected to $m$ satellites.
\end{example}

As in Section \ref{sec:subdivision},
we say that a contraction $\varphi:T(\uv,\un)\to R$ 
{\bf factors nontrivially} if there exists a non-identity morphism $\uf:[\um]\to[\un]$ in $\OIr$
and a contraction $\psi:T(\uv,\um)\to R$ such that $\varphi = \psi\circ\Psi_{T,\uv}(\uf)$.

\begin{proposition}\label{OIrF2}
The sprouting functor $\Phi_{T,\uv}:\OIr \rightarrow \cTop$ has property (F).
\end{proposition}

\begin{proof}
The philosophy of the proof is nearly identical to that of Proposition \ref{OIrF}.
Property (F) says exactly that, for any tree $R$, the set of contractions 
from some $T(\uv,\um)$ to $R$ that do not factor nontrivially is finite.
Let $\psi:T(\uv,\um)\to R$ be given.
We have $$|T(\uv,\um)| = |T| + |\um|,$$ so $\psi$ must contract $|T| + |\um| - |R|$ edges.
If $|\um|$ is sufficiently large, then at least one of those edges must be one of the newly sprouted edges.
We may then factor $\psi$ nontrivially by first contracting that edge.

This tells us that, if we are looking for contractions from some $T(\ue,\um)$ to $R$ that do not factor nontrivially,
we only need to consider finitely many $r$-tuples $\um$.  The proposition then
follows from the fact that all Hom sets in $\cT$ are finite.
\end{proof}

\begin{remark}\label{other-cats-F2}
As in the case of subdivisons (Remark \ref{other-cats-F}), we may define an analogous functor
valued in $\cRTop$, and it will still have property (F).  In contrast with Remark \ref{other-cats-F}, we may {\bf not} define an analogous
functor valued in any bounded leaf category $\cT_l^\op$, since the operation of sprouting yields trees with
arbitrary numbers of leaves.
\end{remark}

\subsection{Generation degree, smallness, and dimension growth}\label{sec:growth}
We say that a module $M$ in $\Rep_k(\cTop)$
is {\bf generated in degrees \boldmath{$\leq d$}}
if there exist trees
$T_1,\ldots,T_r$, each with at most $d$ edges, such that $M$ is isomorphic to a quotient
of $\oplus_{i=1}^r P_{T_i}$.
Equivalently, $M$ is generated in degrees $\leq d$ if, for every tree $T$ with more than $d$ edges,
$M(T)$ is spanned by the images of $\varphi^*$ for various proper contractions
$\varphi:T\to T'$.
We say that $M$ is {\bf \boldmath{$d$}-small} if it is isomorphic
to a subquotient of a module that is generated in degrees $\leq d$, and {\bf \boldmath{$d$}-smallish}
if it admits a filtration whose associated graded is $d$-small.
We make similar definitions for modules in $\Rep_k(\cRTop)$ or $\Rep_k(\cT_l^\op)$. 

\begin{proposition}\label{smallish fg}
If $M$ is $d$-smallish for some $d$, then $M$ is finitely generated.
\end{proposition}

\begin{proof}
Choose a filtration of $M$ such that the associated graded $\operatorname{gr} M$ is $d$-small.
Theorem \ref{Noetherian} implies that $\operatorname{gr} M$ is finitely generated.  This means that there is a finite collection of trees
$T_1,\ldots,T_r$ of trees, along with elements $v_i\in \operatorname{gr} M(T_i)$,
such that, for any tree $T$, the natural map
$$\bigoplus_{i=1}^r \bigoplus_{\varphi:T\to T_i} k\cdot e_{i,\varphi} \to \operatorname{gr} M(T)$$
taking $e_{i,\varphi}$ to $\varphi^* v_i$
is surjective.  For each $i$, choose an arbitrary lift $\tilde v_i\in M(T_i)$ of $v_i$.
Since surjectivity is an open condition, the nautral map
$$\bigoplus_{i=1}^r \bigoplus_{\varphi:T\to T_i} k\cdot e_{i,\varphi} \to M(T)$$
taking $e_{i,\varphi}$ to $\varphi^* \tilde v_i$
is also surjective, which means that $M$ is finitely generated.
\end{proof}

\begin{proof}[Proof of Theorem \ref{bounded growth}]
We may immediately reduce to the case where $M$ is the principal projective $P_R$ for some tree $R$ with $d$ edges.
For any $T$, a contraction from $R$ to $T$ is determined, up to automorphisms of $R$,
by a choice of $|T|-d$ edges of $T$ to contract.  The number of such choices is $\binom{|T|}{d}$, so
$\dim_k P_R(T) \leq |\operatorname{Aut}(R)|\binom{|T|}{d}$.  The fact that we have an inequality rather than an equality
is a reflection of the fact that not every contraction of $T$ with $d$ edges is isomorphic to $R$.
\end{proof}

Theorem \ref{bounded growth} only gives us an upper bound for the dimension of $M(T)$.
We cannot possibly expect equality, since the dimension of $M(T)$ usually depends on the structure of $T$,
not just on the number of edges.  However, if we fix a tree $T$ and an $r$-tuple $\ue$ of distinct directed edges,
we can show that the dimension of $M(T(\ue,\um))$ is eventually equal to a polynomial in $\um$.

\begin{theorem}\label{actual polynomial}
Let $k$ be a field, and suppose that $M$ is $d$-smallish.  Then there exists a multivariate polynomial $f_{M,T,\ue}(t_1,\ldots,t_r)$
of total degree at most $d$ such that, if $\um$ is sufficiently large in every coordinate, 
$$\dim_k M(T(\ue,\um)) = f_{M,T,\ue}(m_1,\ldots,m_r).$$
\end{theorem}

\begin{proof}
Proposition \ref{smallish fg} tells us that $M$ is finitely generated,
though we have no control over the degree of generation.
Theorem \ref{fg} and Proposition \ref{OIrF} combine to tell us that $\Phi^*_{T,\ue}M$ is a finitely generated 
$\OIr$-module.  By \cite[Theorem 6.3.2, Proposition 6.3.3, and Theorem 7.1.2]{sam}, 
this implies that there exists a multivariate polynomial $f_{M,T,\ue}(t_1,\ldots,t_r)$
such that, if $\um\in\N^r$ is sufficiently large in every coordinate, 
$$\dim_k M(T(\ue,\um)) = \dim_k \Phi^*_{T,\ue}M([\um]) = f_{M,T,\ue}(m_1,\ldots,m_r).$$
Theorem \ref{bounded growth} says that $\dim_k M(T(\ue,\um))$ is bounded above by a polynomial of degree $d$ in 
the quantity $|T(\ue,\um)| = |T|-r+|\um|$, thus the total degree of $f_{M,T,\ue}(t_1,\ldots,t_r)$ can be at most $d$.
\end{proof}

\begin{remark}
Let $I_m$ be the standard path of length $m$.  
If $M$ is a $d$-smallish $\cTop$-module, Theorem \ref{actual polynomial}
tells us that the function taking $m$ to $\dim_k M(I_m)$ agrees with a polynomial for sufficiently large $m$.
For positive $m$, the automorphism group of $I_m$ is $S_2$, and if $k$ is a field of characteristic not equal to 2,
we might also guess that the dimensions of isotypic components
of the trivial and sign representations in $M(I_m)$ grow polynomially in $m$.  This, however, is false.
For example, suppose that $M$ is the module that assigns to each tree $T$ the vector space with basis given by the edges of $T$.
More precisely, the principal projective $P_{I_1}$ assigns to each tree the vector space with basis given by the directed edges of $T$,
and we define $M := P_{I_1}^{\Aut(I_1)}$.  The module $M$ is evidently 1-small.  However, the dimension of the trivial isotypic component
of $M(I_m)$ is $\dim_k M(I_m)^{\Aut(I_m)} = \lfloor\frac{m+1}{2}\rfloor$, which is quasi-polynomial in $m$.
\end{remark}

We also have an analogue of Theorem \ref{actual polynomial} in which subdivision is replaced by sprouting.
The proof is identical, so we omit it.
Fix a tree $T$ and an $r$-tuple $\uv$ of distinct vertices.

\begin{theorem}\label{actual polynomial 2}
Let $k$ be a field, and suppose that $M$ is $d$-smallish.  Then there exists a multivariate polynomial $f_{M,T,\uv}(t_1,\ldots,t_r)$
of total degree at most $d$ such that, if $\um$ is sufficiently large in every coordinate, 
$$\dim_k M(T(\uv,\um)) = f_{M,T,\uv}(m_1,\ldots,m_r).$$
\end{theorem}

In the applications that follow, it will be crucial to know that tensor products of small modules behave in the expected way.

\begin{proposition}\label{tensor}
Suppose that $M$ and $N$ are $\cTop$-modules.
If $M$ is $d$-small and $N$ is $e$-small, then $M\otimes_k N$ is $(d+e)$-small.
\end{proposition}

\begin{proof}
We may immediately reduce to the case where $M = P_R$ and $N = P_S$, where $|R| = d$ and $|S| = e$.
Then for any tree $T$, $(M\otimes_k N)(T)$ has a basis given by ordered pairs consisting of a $\cT$-morphism
from $T$ to $R$ and a $\cT$-morphism from $T$ to $S$.  A $\cT$-morphism from $T$ to $R$ contracts $|T|-d$ edges,
and a $\cT$-morphism from $T$ to $S$ contracts $|T|-e$ edges.  For any choice of this pair of morphisms, the number
of edges that are contracted by both morphisms is at least $$(|T|-d) + (|T|-e) - |T| = |T|-d-e,$$ which means that the 
two morphisms both factor through a $\cT$-morphism from $T$ to a tree with at most $d+e$ edges.
\end{proof}

\begin{remark}\label{other-cats-growth}
By Theorem \ref{Noetherian} and Remarks \ref{other-cats-F} and \ref{other-cats-F2},
almost all of the results in this section hold equally well when $\Rep_k(\cTop)$ is replaced by $\Rep_k(\cRTop)$
or $\Rep_k(\cT_l^\op)$.  The only exception is that it is not possible to replace $\Rep_k(\cTop)$ with 
$\Rep_k(\cT_l^\op)$ in Theorem \ref{actual polynomial 2}, since sprouting does not make sense in a bounded leaf subcategory.
\end{remark}

\section{Homology of configuration spaces}\label{sec:homology}
Given a graph $G$, the \textbf{\boldmath{$n$}-stranded unordered configuration space of \boldmath{$G$}} is the topological space
\[
\UConf_n(G) := \big\{(x_1,\ldots,x_n) \in G^n \bigmid x_i \neq x_j\big\}\big{/}S_n.
\]
We will study the homology of these spaces for fixed $n$, with $G$ being either a tree or the cone over a tree.

\subsection{The reduced \'Swi\k{a}tkowski complex}

The primary tool used to compute the homology groups of configuration spaces of graphs is the 
reduced \'Swi\k{a}tkowski complex, originally defined by An, Drummond-Cole, and Knudsen \cite{ADK}. 
Fix a graph $G$, and let $A_G$ to be the integral polynomial ring generated by the edges of $G$. 
A \textbf{half-edge} of $G$ is a pair consisting of a vertex $v$ and an edge $e$ such that $v$ is an endpoint of $e$. 
Given a half-edge $h$, we denote its vertex by $v(h)$ and its edge by $e(h)$. 

For any vertex $v$, let $S(v)$ denote the free $A_G$-module generated by the symbol $\emptyset$
along with all half-edges of $G$ with vertex $v$.
We equip $S(v)$ with a bigrading by defining an edge to have degree $(0,1)$, $\emptyset$ to have degree $(0,0)$,
and a half-edge to have degree $(1,1)$.  Let $\tS(v)\subset S(v)$ be the submodule generated by the elements
$\emptyset$ and $h - h'$ for all half edges $h$ and $h'$.  We equip $\tS(v)$ with an $A_G$-linear differential $\partial_v$
of degree $(-1,0)$ by putting $$\partial (h-h') := \big(e(h)-e(h')\big) \emptyset  \and \partial\emptyset = 0.$$
We then define the
\textbf{reduced \'Swi\k{a}tkowski complex}
\[
\Sw(G) \;\;:=\;\; \bigotimes_{v \in \Vert(G)} \widetilde{S}(v),
\]
where the tensor product is taken over the ring $A_G$; this is a bigraded free $A_G$-module with a differential $\partial$.

For any graph $G$, let $H_\bullet\big(\UConf_\star(G)\big)$ denote the bigraded abelian group
\[
H_\bullet\big(\UConf_\star(G)\big) := \bigoplus_{(i,n)} H_i\big(\UConf_n(G); \Z\big).
\]

\begin{theorem}{\em \cite[Theorem 4.5]{ADK}}\label{swiacompute}
There is an isomorphism of bigraded abelian groups
\[
H_\bullet\big(\UConf_\star(G)\big) \cong H_\bullet\big(\Sw(G)\big).
\]
\end{theorem}

\subsection{Functoriality}
If $\iota:G \rightarrow G'$ is a simplicial embedding of graphs, then one obtains a natural pushforward map 
$$\iota_*: H_i\big(\UConf_n(G); \Z\big)\to H_i\big(\UConf_n(G'); \Z\big),$$
along with a natural lift to a map 
of differential bigraded modules \cite[Section 4.2]{ADK}
$$\widetilde\iota_*: \Sw(G)\to\Sw(G').$$
What is less obvious is that, if $\varphi$ is a contraction, then there is a natural map of differential bigraded modules
\[
\widetilde\varphi^\ast:\Sw(G') \rightarrow \Sw(G),
\]
which induces a map $$\varphi_*: H_i\big(\UConf_n(G'); \Z\big)\to H_i\big(\UConf_n(G); \Z\big)$$
by passing to homology \cite[Lemma C.7]{ADK}.

To describe $\widetilde\varphi^\ast$, we first consider the case where the number of edges of $G$ is one greater than the number of edges of $G'$;
we call such a contraction $\varphi$ a {\bf simple contraction}.
We identify the unique edge of $G$ that is contracted by $\varphi$ with the interval $[0,1]$.
Let $h_0$ (respectively $h_1$) be the half edge of $G$ consisting of the vertex $0$ (respectively $1$) and the edge $[0,1]$.
Let $w'\in G'$ be the image of the edge $[0,1]$.
Each edge of $G'$ is mapped to isomorphically by a unique edge of $G$, and similarly for half edges.  
This gives us a canonical ring homomorphism $A_{G'} \to A_G$ along with an $A_{G'}$-module homomorphism
$$\bigotimes_{v' \in \Vert(G')\smallsetminus\{w'\}} \widetilde{S}(v') \;\;\;\to \bigotimes_{v \in \Vert(G)\smallsetminus\{0,1\}} \widetilde{S}(v).$$
Given a half edge $h'$ of $G'$ with $v(h') = w'$, let $h$ be the unique half edge of $G$ mapping to $h'$.
We then define an $A_{G'}$-module homomorphism $$\widetilde{S}(w')\to \widetilde{S}(0)\otimes \widetilde{S}(1)$$
by the formula $$\emptyset\mapsto\emptyset\otimes\emptyset\and h'\mapsto
\begin{cases}
(h-h_0) \otimes \emptyset\;\;\;\text{if $v(h)=0$}\\
\emptyset \otimes (h-h_1)\;\;\;\text{if $v(h)=1$.}
\end{cases}$$
Tensoring these two maps together, we obtain the homomorphism $\widetilde\varphi^\ast:\Sw(G') \rightarrow \Sw(G)$,
and it is straightforward to check that this homomorphism respects the differential.
Arbitrary contractions may be obtained as compositions of simple contractions, and the induced homomorphism is independent of choice of factorization into simple contractions.
To summarize, we have the following result.

\begin{theorem}\label{summary}{\em \cite{ADK}}
There is a bigraded differential $\cTop$-module that assigns to each tree $T$ the reduced \'Swi\k{a}tkowski complex $\Sw(T)$.
The homology of this bigraded differential $\cTop$-module is the bigraded $\cTop$-module that assigns to each tree $T$
the bigraded Abelian group $H_{\bullet}\big(\UConf_{\star}(T)\big)$.
\end{theorem}

% To summarize, we have a bigraded $\cTop$-module $H_{\bullet}\big(\UConf_{\star}(-)\big)$ which is isomorphic, via Theorem \ref{swiacompute}
% to the homology of the bigraded differential $\cTop$-module $\Sw(-)$.

\subsection{Configuration spaces of trees}
The purpose of this section is to prove Theorem \ref{treefing}.

\begin{proof}[Proof of Theorem \ref{treefing}]
Given a tree $T$ and a pair of natural numbers $i$ and $n$, let $\Sw(T)_{i,n}$ be the degree $(i,n)$ summand of the 
reduced \'Swi\k{a}tkowski complex. 
We will show that the $\cTop$-module taking a tree $T$ to the abelian group $\Sw(T)_{i,n}$
is generated in degrees $\leq n+i$.  Smallness will then follow from Theorem \ref{summary},
and finite generation from Theorem \ref{Noetherian}.

% Let $T$ be given with $|T|\geq n+2i$.
The group $\Sw(T)_{i,n}$ is generated by elements of the form 
$$\sigma := e_1\cdots e_{n-i}\;\; \bigotimes_{j=1}^{i}\, (h_{j0} - h_{j1})\;\;\; 
\otimes \bigotimes_{v\notin\{v_1,\ldots,v_i\}} \!\!\!\emptyset,$$
where $e_1,\ldots,e_{n-i}$ are edges (not necessarily distinct), $v_1,\ldots,v_i$ are vertices (distinct), and for each $j$,
$h_{j0}$ are $h_{j1}$ are half edges at the vertex $v_j$.
For a particular $\sigma$ of this form, we will call $\{v_1,\ldots,v_i\}$ the set of {\bf distinguished vertices}.
Without loss of generality, we may assume that there is some integer $r$ with $0\leq r\leq i$ such that $v_j$
is adjacent to some other distinguished vertex if and only if $j\leq r$.  We may also assume that, if $j\leq r$,
$e(h_{j1})$ connects $v_j$ to another distinguished vertex; if not, then $\sigma$ may be written as a difference
of classes of this form.

We call an edge $e$ a {\bf distinguished edge} if one of the following four conditions hold:
\begin{itemize}
\item $e$ connects two distinguished vertices
\item $e = e_k$ for some $k\leq n-i$
\item $e = e(h_{j0})$ for some $j\leq i$
\item $e = e(h_{j1})$ for some $j\leq i$.
\end{itemize}
We claim that there are at most $n+i$ distinguished edges.
Indeed, there are at most $r$ edges that connect two distinguished vertices\footnote{This is because the induced subgraph on $v_1,\ldots,v_r$ is a forest;
equality is attained iff $r=0$.}
and these include $e(h_{j1})$ for every $j\leq r$.  This means that the maximum possible number of distinguished
edges is $r + (n-i) + i + (i-r) = n+i$.

Let $T$ be given with $|T| > n+i$.  Since there are at most $n+i$ distinguished edges, we may choose an edge $e$
which is not distinguished.
Let $T' := T/e$ be the tree obtained from $T$ by contracting $e$, and let $\varphi:T\to T'$ be the canonical simple
contraction.  
Let $e_k'$ be the image of $e_k$ in $T'$, $v_j'$ the image of $v_j$ in $T'$, $h_{j0}'$
the image of $h_{j0}$ in $T'$, and $h_{j1}'$ the image of $h_{j1}$ in $T'$.  Let
$$\sigma' := e'_1\cdots e'_{n-i}\;\; \bigotimes_{j=1}^{i}\, (h_{j0}' - h_{j1}')\;\;\; 
\otimes \bigotimes_{v'\notin\{v'_1,\ldots,v'_i\}} \!\!\!\emptyset\;\;\;\in\;\;\;\Sw(T')_{i,n}.$$
We claim that $\sigma=\widetilde\varphi^*\sigma'$.

If $e$ is not incident to any vertex $v_j$, this is obvious.
The interesting case occurs when $e$ is incident to one of the distinguished
vertices.  Assume without loss of generality that it is incident to $v_1$, and let $w$ be the other end point of $e$.
% Note that, since $e$ is not distinguished, neither is $w$.
Let $h$ be the half edge of $T$ with $e(h) = e$ and $v(h) = v_1$.
Applying the map $\varphi^*$ replaces each $e_k'$ with $e_k$.  When $j>1$, it replaces $h_{j0}'$ with $h_{j0}$
and $h_{j1}'$ with $h_{j1}$.  It replaces $h_{10}'$ with $h_{10} - h$ and $h_{11}'$ with $h_{11}-h$.
This means that it replaces $h_{j0}' - h_{j1}'$ with $h_{j0} - h_{j1}$, and therefore that 
$\widetilde\varphi^*\sigma' = \sigma$.

We thus conclude that every element of $\Sw(T)_{i,n}$ is a linear combination
of elements in the images of map associated with simple contractions; this completes the proof.
\end{proof}

\begin{remark}
Chettih and L\"utgehetmann prove that homology groups of ordered configuration spaces
of trees are generated by products of what they call {\bf basic classes} \cite[Theorem A]{CL}.
One can produce an alternative proof of Theorem \ref{treefing} by using this result, along with the coinvariant map
that takes the homology of an ordered configuration space surjectively only the homology of the corresponding
unordered configuration space, sending basic classes to the {\bf star classes} of \cite{ADK}.
We prefer to work with the reduced \'Swi\k{a}tkowski complex because the proof of Theorem \ref{treefing}
serves as a model for the proof of Theorem \ref{conetreefing}, where we will not have any analogue of the 
Chettih--L\"utgehetmann result available to us.
\end{remark}

\subsection{Cones over trees}\label{sec:cones}
For any tree $G$, we define the {\bf cone over \boldmath{$G$}} to be the graph $\cone(G)$ obtained from $G$ 
by adding one new vertex $p$ along with an edge connecting $p$ to each of the original vertices.
More precisely, $$\Vert(\cone(T)) := \Vert(T)\sqcup\{p\}\and \Edge(\cone(T)) := \Edge(T)\sqcup\{e_v\mid v\in\Vert(T)\},$$
where $e_v$ is an edge from $v$ to $p$.

Suppose that $\varphi:T\to T'$ is a contraction.  We would like to say that $\varphi$ induces a contraction from $\cone(T)$
to $\cone(T')$, but this is not quite right.  Indeed, contractions are by definition homotopy equivalences, and $\cone(T)$ cannot
be homotopy equivalent to $\cone(T')$ unless $\varphi$ is an isomorphism.  
Instead, what happens is that $\varphi$ induces a contraction from $\pi_\varphi:\cone(T)\to G_\varphi$, where $G_\varphi$ defined by setting
$$\Vert(G_\varphi) := \Vert(T') \sqcup \{p'\} \and \Edge(G_\varphi) := \Edge(T') \sqcup \{e'_w\mid w\in \Vert(T)\},$$
where $e_w$ is an edge from $\varphi(w)$ to $p'$.  In particular, the number of edges connecting a vertex $w'$ to $p'$
is equal to the number of vertices in the preimage of $w'$.

We get around this technical difficulty by working with rooted trees.
Let $(T,v)$ be a rooted tree, and consider the partial order on $\Vert(T)$ introduced in Section \ref{sec:trees}.
Let $\varphi:(T,v)\to(T',v')$ be a contraction.  For each vertex $w'$ of $T'$, there is a unique maximal vertex $w\in\varphi^{-1}(w')$.
This allows us to define an embedding $\iota^\varphi:\cone(T')\to G_\varphi$ by sending $e_{w'}$ to $e'_w$.

\begin{example}
Suppose that $\varphi$ is the contraction of rooted trees depicted in the figure on the left in Example \ref{Barter}.
Then $G_\varphi$ is the graph shown below, where $T'$ is the vertical edge and $p'$ is the vertex on the right.
The embedding $\iota^\varphi$ identifies the cone over $T'$ with the triangle
obtained by deleting the lower of the two curved edges.
\begin{figure}[ht]
\begin{center}
\includegraphics[width=1cm]{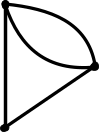}
\end{center}
% \caption{}
\end{figure}
\end{example}

The embedding $\iota_\varphi$ induces a map $$\iota^\varphi_*:H_i\big(\UConf_n(\cone(T')); \Z\big)\to H_i\big(\UConf_n(G_\varphi); \Z\big)$$
and the contraction $\pi_\varphi$ induces a map $$\pi_\varphi^*:H_i\big(\UConf_n(G_\varphi); \Z\big)\to H_i\big(\UConf_n(\cone(T)); \Z\big),$$
and we define $$\varphi^* := \pi_\varphi^*\circ \iota^\varphi_*:H_i\big(\UConf_n(\cone(T')); \Z\big)\to H_i\big(\UConf_n(\cone(T)); \Z\big).$$
Similarly, we define
$$\widetilde\varphi^* := \widetilde\pi_\varphi^*\circ \widetilde\iota^\varphi_*:\Sw(\cone(T'))\to \Sw(\cone(T)).$$

We can then state the following analogue of Theorem \ref{summary}.

\begin{theorem}\label{cone-summary}
There is a bigraded differential $\cRTop$-module that assigns to each rooted tree $(T,v)$ the reduced \'Swi\k{a}tkowski complex $\Sw(\cone(T))$.
The homology of this bigraded differential $\cTop$-module is the bigraded $\cRTop$-module that assigns to each rooted tree $(T,v)$
the bigraded Abelian group $H_{\bullet}\big(\UConf_{\star}(T)\big)$.
\end{theorem}

\begin{remark}
We note that, for both of the $\cRTop$-modules in the statement of Theorem \ref{cone-summary}, the Abelian group assigned to a rooted tree $(T,v)$
depends only on the tree $T$, but the homomorphisms between these groups depend on the root.
\end{remark}

\subsection{Configuration spaces of cones over trees}
The purpose of this section is to prove Theorem \ref{conetreefing}.

\begin{proof}[Proof of Theorem \ref{conetreefing}]
The proof of this theorem is similar to the proof of Theorem \ref{treefing}, but there are some subtle differences in the details.
Given a tree $T$ and a pair of natural numbers $i$ and $n$, let $\Sw(\cone(T))_{i,n}$ be the degree $(i,n)$ summand of the 
reduced \'Swi\k{a}tkowski complex. 
We will show that the $\cRTop$-module taking a rooted tree $(T,v)$ to the abelian group $\Sw(T)_{i,n}$
is generated in degrees $\leq n+i$.  Smallness will then follow from Theorem \ref{cone-summary},
and finite generation from Theorem \ref{Noetherian}.

The group $\Sw(\cone(T))_{i,n}$ is generated by classes of form 
$$\sigma := e_1\cdots e_{n-i}\;\; \bigotimes_{j=1}^{i}\, (h_{j0} - h_{j1})\;\;\; 
\otimes \bigotimes_{v\notin\{v_1,\ldots,v_i\}} \!\!\!\emptyset,$$
where $e_1,\ldots,e_{n-i}$ are edges of $\cone(T)$ (not necessarily distinct), $v_1,\ldots,v_i$ are vertices of $\cone(T)$ (distinct), and, for each $j$,
$h_{j0}$ are $h_{j1}$ are half edges of $\cone(T)$ at the vertex $v_j$.
For a particular $\sigma$ of this form, we will call $\{v_1,\ldots,v_i\}$ the set of {\bf distinguished vertices}.
Here is the first difference between this proof and the proof of Theorem \ref{treefing}:  we may assume that, for all $j$, $e(h_{j1}) = e_{v_j}$,
the edge connecting $v_j$ to the cone point $p$.

Let us first treat the case where $p$ is not one of the distinguished vertices.
We call an edge $e$ of $T$ a {\bf distinguished edge} if one of the following three conditions hold:
\begin{itemize}
\item $e$ connects two distinguished vertices
\item $e = e_k$ for some $k\leq n-i$
\item $e = e(h_{j0})$ for some $j\leq i$.
\end{itemize}
Note that distinguished edges are edges of $T$, not of $\cone(T)$.  We claim that there are at most $n+i$ distinguished edges.
Indeed, there are at most $i$ edges that connect two distinguished vertices (with equality iff $i=0$), 
so the maximum possible number of distinguished edges is $i+(n-i)+i = n+i$.

Next, let's treat the case where $p$ is one of the distinguished vertices.  Without loss of generality, we will assume that $p = v_i$.
Let $w_0,w_1\in\Vert(T)$ be the endpoints of the edges $e(h_{i0})$ and $e(h_{i1})$.  
We then call an edge $e$ of $T$ 
a {\bf distinguished edge} if one of the following five conditions hold:
\begin{itemize}
\item $e$ connects two distinguished vertices of $T$
\item $e = e_k$ for some $k\leq n-i$
\item $e = e(h_{j0})$ for some $j\leq i-1$
\item $e$ connects $w_0$ to a vertex that is greater than $w_0$ in the partial order on $\Vert(T)$
\item $e$ connects $w_1$ to a vertex that is greater than $w_1$ in the partial order on $\Vert(T)$.
\end{itemize}
We again claim that there are at most $n+i$ distinguished edges.
Indeed, there are at most $i-1$ edges that connect two distinguished vertices (with equality iff $i=1$).
Furthermore, every vertex of $T$ has a unique cover in the partial order, so there is at most one edge connecting
$w_0$ to a vertex greater than $w_0$ (with equality iff $w_0$ is not the root), and likewise for $w_1$.  Thus the maximum
possible number of distinguished edges is $(i-1)+(n-i)+(i-1)+1+1 = n+i$.

We now proceed as in the proof of Theorem \ref{treefing}.
Let $T$ be given with $|T| > n+i$.  Since there are at most $n+i$ distinguished edges, we may choose an edge $e$
which is not distinguished.
Let $T' := T/e$ be the tree obtained from $T$ by contracting $e$, and let $\varphi:T\to T'$ be the canonical simple
contraction.  We define $\sigma'$ as before, and we claim that $\sigma=\widetilde\varphi^*\sigma'$.

The argument is basically the same, but there is a new subtlety that arises if $p=v_i$ and either $w_0$ or $w_1$ is one of the endpoints of $e$.
Assume without loss of generality that $w_0$ is an endpoint of $e$, and let $u$ be the other endpoint.  Since $e$ is not a distinguished edge,
we have $u<w_0$ in the partial order on $\Vert(T)$, and this is exactly the condition that we need to ensure that $\sigma=\widetilde\varphi^*\sigma'$.
\end{proof}

\begin{remark}\label{biconnected}
For any tree $T$, $\cone(T)$ is biconnected, which implies that $H_1\big(\UConf_n(\cone(T))\big)$ is canonically
isomorphic to $H_1\big(\UConf_2(\cone(T))\big)$ \cite[Lemma 3.12]{KP}.  In particular, this means that 
the $R\cTop$-module
\[
(T,v) \mapsto H_1\big(\UConf_n(\cone(T))\big)
\]
is in fact $3$-small.
\end{remark}

\subsection{Examples}\label{config-examples}

We next give a number of examples to illustrate Theorems \ref{treefing} and \ref{conetreefing}.

\begin{example}
Consider the tree $K_{m,1}$ with one central vertex connected to $m$ satellites, obtained from sprouting at a single isolated vertex.
Theorems \ref{actual polynomial 2} and \ref{treefing}
together imply that $\dim H_1\big(\UConf_n(K_{m,1}); \Q\big)$ is a polynomial in $m$ of degree at most $n+1$.
It was proved independently in \cite[Theorem 2.6]{Gh-robotics} and \cite[Theorem 0.1]{Sw-graphs}
that the space $\UConf_n(G)$ is homotopy equivalent to a simplicial complex of dimension equal to the number
of vertices of $G$ of degree at least 3.  In particular, $\UConf_n(K_{m,1})$ is homotopy equivalent to a graph.
For any $m$, we have \cite[Theorem 2]{Gal}
\[
\sum_{n\geq 0} \chi\big(\UConf_n(K_{m,1})\big) t^n = \frac{1-(m-1)t}{(1-t)^m}.
\]
From these two facts, we can deduce that
\[
\dim H_1\big(\UConf_n(K_{m,1}); \Q\big) = 1 - \binom{m-1+n}{n} + (m-1)\binom{m-2+n}{n-1},
\]
which is in fact a polynomial in $m$ of degree $n$. Note that the computation of the first Betti numbers of configuration spaces of star graphs has appeared in various places throughout the literature, including \cite{Gh-robotics,ADK,MaSa,Ramos-Graph,FS}.
\end{example}

\excise{
\begin{example}
Fix natural numbers $n$ and $i$. 
For any natural number $m$, let $I_m$ be the standard path of length $m$ on the vertex set $\{0,1,\ldots,m\}$ with root 0.
Then the cone over $I_m$ is a fan graph, and 
Theorems \ref{actual polynomial}, Remark \ref{other-cats-growth}, and Theorem \ref{conetreefing}
combine to say that the $i^\text{th}$ Betti number of the configuration space $\UConf_n(\cone(I_m))$ is eventually a polynomial in $m$
of degree at most $n+i$.
\end{example}
}

\begin{example}
Next, consider the graph $\cone(K_{m,1})$.  
The graph $\cone(K_{m,1})$ has two vertices of degree greater than 2, so
the configuration space of $\cone(K_{m,1})$ is homotopy equivalent to a 2-dimensional simplicial complex. 
Theorems \ref{actual polynomial 2} and \ref{conetreefing}, 
along with Remarks \ref{other-cats-growth} and \ref{biconnected}, together imply that 
$\dim H_1\big(\UConf_n(\cone(K_{m,1})); \Q \big)$
is a polynomial in $m$ of degree at most $3$ and $\dim H_2\big(\UConf_n(\cone(K_{m,1})); \Q \big)$
is a polynomial in $m$ of degree at most $n+2$.
For any $m$, we have \cite[Lemma 3.14]{KP} 
\[
\dim H_1\big(\UConf_n(\cone(K_{m,1})); \Q \big) = \binom{m+1}{2},
\]
which is in fact a polynomial in $m$ of degree 2.
We also have \cite[Theorem 2]{Gal}
\[
\sum_{n\geq 0} \chi\big(\UConf_n(\cone(K_{m,1}))\big) t^n = \frac{(1-mt)^2}{(1-t)^{m+1}},
\]
which implies that
\[
\chi\big(\UConf_n(\cone(K_{m,1}))\big) = \binom{m+n}{n} + 2m \binom{m+n-1}{n-1} + m^2  \binom{m+n-2}{n-2}.
\]
Thus we conclude that
$\dim H_2\big(\UConf_n(\cone(K_{m,1})); \Q \big)$ is in fact a polynomial in $m$ of degree $n$.
\end{example}

\begin{example}
More generally, the techniques of Ko and Park will allow us to compute the first Betti numbers of arbitrary cones over trees. 
Remark \ref{biconnected} tells us that these Betti numbers should be bounded above by a cubic polynomial in $|T|$,
and that they should grow as a polynomial of degree at most 3 when we subdivide or sprout.

Fix a vertex $v$ of $T$ and write $\{T_i\}_{i = 1}^{\deg(v)}$ for the $\deg(v)$ subtrees of $T$ obtained by taking the closures of the connected components after removing $v$. If we write $b_{\cone(T)}(n)$ for the first Betti number of $\UConf_n(\cone(T))$, then \cite[Lemma 3.14]{KP} tells us that
\[
b_{\cone(T)}(n) = \left(\sum_i b_{\cone(T_i)}(n)\right) + \binom{\deg(v)}{2}.
\]
If $T$ is homeomorphic to a line segment, then $\cone(T)$ is a circle and $b_{\cone(T)}(n) = 1$. 
Applying the above recursion for every vertex, we find that
\[
b_{\cone(T)}(n) \;\;=\;\; |T| \;\;+ \sum_{v\in\Vert(T)} \binom{\deg(v)}{2}.
\]
Note that this expression is bounded by a polynomial in $|T|$ of degree 2, since
\[
\sum_{v\in\Vert(T)} \binom{\deg(v)}{2} \;\; \leq\;\; \frac{1}{2}\sum_{v\in\Vert(T)} \deg(v)^2\;\;  \leq\;\;  \frac 1 2 \left(\sum_{v\in\Vert(T)}\deg(v)\right)^2 
=\;\; 2|T|^2.
\]
One also observes from this expression that $b_{\cone(T)}(n)$ grows linearly under subdivision and quadratically under sprouting.
\end{example}

\section{Kazhdan-Lusztig coefficients}\label{sec:KL}
Let $G$ be a graph.  Let $R_G$ be the $\C$-subalgebra of rational functions in $\{x_v\mid v\in\Vert(G)\}$
generated by the elements $\left\{\frac{1}{x_v-x_w}\Bigmid v\neq w\;\text{adjacent}\right\}$, and let $X_G := \Spec R_G$.
The ring $R_G$ is called the {\bf Orlik-Terao algebra} of $G$ and the variety $X_G$ is called the {\bf reciprocal plane} of $G$.
We will be interested in the intersection homology group $\IH_{2i}(X_G)$, which is a complex vector space whose dimension
is equal to the coefficient of $t^i$ in the Kazhdan-Lusztig polynomial of the matroid associated with $G$ \cite[Theorem 3.10 and Proposition 3.12]{EPW}.

If $\varphi:G\to G'$ is a contraction, we obtain a canonical map $\varphi^*:\IH_{2i}(X_{G'})\to\IH_{2i}(X_G)$, and these maps compose in the expected way
\cite[Theorem 3.3(1,3)]{fs-braid}.  The matroid associated with any tree is Boolean and Boolean matroids have trivial Kazhdan-Lusztig polynomials 
\cite[Corollary 2.10]{EPW}, so we do not obtain interesting $\cTop$-modules by letting $G$ be a tree.  However, by letting $G$ be the cone
over a tree, we find many examples of graphs with interesting Kazhdan-Lusztig coefficients, including fan graphs \cite{fan-wheel-whirl}
and thagomizer graphs \cite{thag}.  The purpose of this section is to study the corresponding $\cTop$-modules.

\begin{remark}\label{simplification}
If $G$ is any graph, we define the {\bf simplification} of $G$ to be the graph obtained by deleting loops
and identifying any two edges with the same end points.  It is clear from the definition of $R_G$ that the ring $R_G$, the variety $X_G$,
and the vector space $\IH_{2i}(X_G)$ do not change when $G$ is replaced by its simplification.  This stands in stark contrast to homology
groups of configuration spaces, and explains why we will not need to work with rooted trees (see Remark \ref{more or less}).
\end{remark}

\subsection{The spectral sequence}\label{sec:ss}
The intersection homology group $\IH_{2i}(X_G)$ can be computed by means of a certain spectral sequence, 
which we now describe.
For any graph $G$, let $\OS^\bullet(G)$ be the {\bf Orlik-Solomon algebra} of the matroid associated with $G$.
For any natural number $d$, we will denote the linear dual of $\OS^d(G)$ by $\OS_d(G)$.
For the purposes of this paper, we will need to know five things about the Orlik-Solomon algebra:
\begin{itemize}
\item $\OS^1(G)$ is spanned by classes $\{x_e\mid e\in\Edge(G)\}$, subject to the relations
that $x_e=0$ if $e$ is a loop and $x_e = x_{e'}$ if $e$ and $e'$ have the same endpoints.
\item $\OS^\bullet(G)$ is generated as an algebra by $\OS^1(G)$.
\item The Orlik-Solomon algebra of a graph is canonically isomorphic to that of its simplification.
\item If $G'$ is a contraction of $G$, we obtain a canonical map $\OS^\bullet(G)\to\OS^\bullet(G')$ 
by killing the generators indexed by contracted edges.  This in turn induces a map $\OS_\bullet(G')\to \OS_\bullet(G)$.
\item If $G$ is the disjoint union of $G_1$ and $G_2$ or if $G$ is obtained by gluing $G_1$ and $G_2$
along a single vertex, then $\OS^\bullet(G) \cong \OS^\bullet(G_1)\otimes \OS^\bullet(G_2)$.
\end{itemize}

A {\bf flat} of $G$ is a subgraph $F\subset G$ with the same
vertex set and that property that, if $F$ contains all but one edge of some cycle in $G$, then it contains the last edge, as well.
If $F$ is a flat, we define $G/F$ to be
the graph obtained by simultaneously contracting all of the edges in $F$.\footnote{It is slightly confusing to note that, if $F$ contains
a cycle, the natural graph morphism from $G$ to $G/F$ is not a contraction in the sense of Section \ref{sec:trees} because it is not
a homotopy equivalence.  We will refrain from using the word ``contraction'' in any sense other than that in which we have defined it.}
  The {\bf rank} of $F$ is equal to the number of vertices
minus the number of connected components, and the {\bf corank} of $F$, denoted $\crk F$, is
the rank of $G$ minus the rank of $F$.

\begin{theorem}\label{spectral}{\em \cite[Theorems 3.1 and 3.3]{fs-braid}}
For any graph $G$ and positive integer $i$, there is a first quadrant homological spectral sequence $E(G,i)$ converging to $\IH_{2i}(X_G)$,
with $$E(G,i)^1_{p,q} = \bigoplus_{\crk F = p}\OS_{2i-p-q}(F) \otimes \IH_{2(i-q)}(X_{G/F}).$$
If $\varphi:G\to G'$ is a contraction, there is a canonical map $\varphi^*:E(G',i)\to E(G,i)$ of spectral sequences, composing in the expected way, and converging to the aforementioned map $\IH_{2i}(X_{G'})\to \IH_{2i}(X_G)$.
The map $E(G',i)^1_{p,q}\to E(G,i)^1_{p,q}$ kills the $F$-summand unless $F$ contains all of the contracted edges.  
In this case, the image of $F$
in $G'$ is a flat $F'$ of $G'$, and $G'/F'$ is canonically isomorphic to $G/F$.  The map takes
the $F$-summand of $E(G,i)^1_{p,q}$ to the $F'$-summand of $E(G',i)^1_{p,q}$ by 
the canonical map $\OS_{2i-p-q}(F)\to\OS_{2i-p-q}(F')$ tensored with the identity map on $\IH_{2(i-q)}(X_{G/F})$.
\end{theorem}

\subsection{Orlik-Solomon algebras of trees and their cones}
Suppose that $\varphi:T\to T'$ is a contraction of trees.  Since the Orlik-Solomon algebra is functorial with respect to contractions,
we have a $\cTop$-module $\OS_d$ that takes a tree $T$ to $\OS_d(T)$.
Recall from Section \ref{sec:cones} that $\varphi$ induces a contraction from $\cone(T)$ to $G_\varphi$, where $G_\varphi$
is a graph whose simplification is canonically isomorphic to $\cone(T')$.  Since the Orlik-Solomon algebra of a graph is canonically
isomorphic to that of its simplification, this tells us that we also 
have a $\cTop$-module $\OS^\cone_d$ that takes a tree $T$ to $\OS_d(\cone(T))$.

\begin{proposition}\label{OSi}
The $\cTop$-modules $\OS_d$ and $\OS^{\cone}_d$ are $d$-small for all $d\in\N$.
\end{proposition}

\begin{proof}
Recall from Example \ref{paths} that we have defined $I_m$ to be the standard path of length $m$ on the vertex set $\{0,\ldots,m\}$.
We have $\OS_0 = \OS^{\cone}_0 = P_{I_0}$, which proves that $\OS_0$ and $\OS^{\cone}_0$ are both 0-small.  Since $\OS^\bullet$ 
and $\OS_{\cone}^\bullet$ are generated as algebras in degree 1,
Lemma \ref{tensor} implies that it is sufficient to prove that $\OS_1$ and $\OS_1^{\cone}$ are 1-small.

The module $\OS_1$ associates to any tree a vector space with basis given by its edges.  In particular, $\OS_1(I_1) = \C\cdot x_{01}$.
If $e$ is an edge of $T$ and $\varphi:T\to I_1$ is a morphism that contracts every edge except for $e$, then $\varphi^*x_{01} = x_e$.  This shows
that $\OS_1$ is generated in degree 1 and therefore 1-small.

The edges of the cone over a tree are in bijection with the edges and vertices of the tree, so
the module $\OS_1^{\cone}$ associates to any tree a vector space with basis $\{x_e\}\sqcup\{x_v\}$ indexed by edges and vertices.
In particular, $\OS_1^{\cone}(I_1) = \C\{x_{01},x_0,x_1\}$.
Let $T$ be a tree and $e$ an edge of $T$ with vertices $v$ and $w$.
Consider the unique morphism $\psi:T\to I_1$ that sends $v$ to $0$ and $w$ to $1$.
Then $$\varphi^*x_{01} = x_e \and \varphi^*x_0 = \sum_{\varphi(v')=0} x_{v'}.$$
It is clear that classes of this form span $\OS_1^{\cone}(T)$, hence $\OS_1^{\cone}$ is generated in degree 1 and therefore 1-small.
\end{proof}

\subsection{Flats of cones over trees}\label{sec:flats}
Fix a tree $R$.  For any tree $T$, let $\Comp_R(T)$ be the set of 
ways to break $T$ up into a collection of disjoint subtrees, indexed by the vertices of $R$,
with adjacency of subtrees determined by adjacency in $R$.
More precisely, an element of $\Comp_R(T)$ is a tuple $\underline{U} = \big(U_v \mid v\in \Vert(R))\big)$ of subtrees of $T$
such that 
$$\Vert(T) \;\; = \bigsqcup_{v\in \Vert(R)} \Vert(U_v)$$
and $U_v$ is adjacent to $U_w$ in $T$ if and only if $v$ is adjacent to $w$ in $R$.

We will say that a subset $W\subset\Vert(R)$ is {\bf groovy} if it has the property that every edge
of $R$ is incident to at least one vertex of $W$.
Let $F(T)$ be the set of triples
$(R,W,\underline{U})$, where $R$ is a tree, $\underline{U}\in\Comp_R(T)$,
and $W\subset \Vert(R)$ is groovy.  We say that two triples $(R,W,\underline{U})$ and $(R',W',\underline{U}')$
are {\bf equivalent} if there is an isomorphism from $R$ to $R'$ taking $W$ to $W'$ and $\underline{U}$
to $\underline{U}'$.  

Given a triple $(R,W,\underline{U}) \in F(T)$, we may construct a flat of $\cone(T)$
by taking the edges of $U_v$ for all $v\in \Vert(R)$ along with the edges connecting $v$ to the cone point for all $v\in W$.
Every flat arises in this manner, and two elements of $F(T)$ give rise to the same flat if and only if they are equivalent.
The contraction of $\cone(T)$ along this flat is isomorphic to $\cone(R_W)$, where $R_W$ is the induced forest on the vertex set $W$.

\begin{example}\label{cone-flat}
Let $R$ be a path of length 3 and $T$ a path of length 5.  An element of $\Comp_R(T)$ is a way to break the 6 vertices of $T$ into 4
blocks, each of which consist of adjacent vertices.  In the picture below, we show an element of $\Comp_R(T)$ consisting of blocks
of sizes 2, 1, 2, and 1 (reading from left to right).  We also select the groovy subset $W\subset \Vert(R)$ consisting of the first and last vertex,
which means that the first and last block get connected to the cone point in the corresponding flat of $\cone(T)$.  We denote this flat by thickened edges.
\begin{figure}[ht]
\begin{center}
\includegraphics[width=8cm]{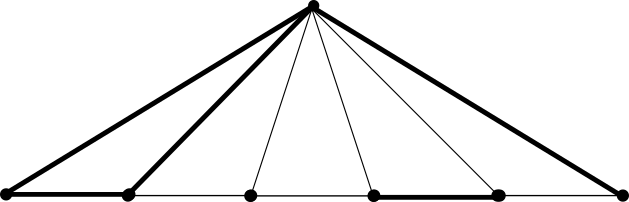}
\end{center}
\end{figure}
\end{example}

Let $\OS^\bullet(R,W,\underline{U})$ denote the Orlik-Solomon algebra of the flat associated with the triple $(R,W,\underline{U})$.
By the discussion of Orlik-Solomon algebras in Section \ref{sec:ss}, this is isomorphic to
$$\bigotimes_{v\in W}\OS^\bullet(U_v)\;\otimes \bigotimes_{v\notin W}\OS^\bullet\big(\cone(U_v)\big).$$

The following lemma is an analogue of \cite[Lemma 4.2]{fs-braid}.  

\begin{lemma}\label{putting them together}
Suppose that we have a tree $R$, a collection of $\cTop$-modules $\big(N_v \mid v\in \Vert(R)\big)$,
and a collection of natural numbers $\big(d_v \mid v\in \Vert(R)\big)$ such that
$N_v$ is $d_v$-small for all $v $.
Let $d = |R| + \sum d_v$.
Define a $\cTop$-module $N$ by the formula
$$N(T)\;\; := \bigoplus_{\underline{U}\in\Comp_R(T)} \;
\bigotimes_{v\in\Vert(R)} N_v(U_v).$$
Then $N$ is $d$-small.
\end{lemma}

\begin{proof}
We may immediately reduce to the case where
$N_v = P_{S_v}$ for some tree $S_v$ with $|S_v| = d_v$.
Then for any tree $T$ and any $\underline{U}\in\Comp_R(T)$,
the corresponding summand of $N(T)$ has a basis given by tuples of contractions from $U_v$ to $S_v$.
Such a map contracts $|U_v| - d_v$ edges.  All together, the number of edges that get contracted
is $\sum \big(|U_v| - d_v\big) = |T| - |R| - \sum d_v$, which means that the tuple of maps factors through
a tree with $|R| + \sum d_v$ edges.
\end{proof}

\subsection{A leaf lemma}
The following technical lemma and corollary will be important in the next section.

\begin{lemma}\label{leaves}
Let $T$ be a tree with at most $l\geq 2$ leaves, and suppose that $Y\subset \Vert(T)$
has the property that every edge of $T$ has exactly one vertex in $Y$.  Then $|T| \leq 2|Y| + l - 2$.
\end{lemma}

\begin{proof}
We proceed by induction on $l$.  If $l=2$, the statement is clear.  Now assume that the statement holds for $l$,
and let $T$ be a tree with $l+1$ leaves.
Choose a leaf of $T$, and consider the path from that leaf to the nearest vertex of degree greater than 2.  Let $k$ be the length of the path.
Let $T'$ be the tree obtained from $T$ by deleting that path, and let $Y'$ be the subset of $Y$ that lies in $T'$.  Then
$T'$ has $l$ leaves, $|T| - |T'| = k$, and $|Y| - |Y'|$ is either $\lfloor k/2\rfloor$ or $\lceil k/2\rceil$, depending on whether or not
the leaf is in $Y$.  In particular, $2|Y'| \leq 2|Y| - k + 1$.
Then
$$|T| = k + |T'| \leq k + 2|Y'| + l - 2 \leq k + 2|Y| - k + 1 + l - 2 = 2|Y| + (l+1) - 2,$$
which completes the proof.
\end{proof}

\begin{corollary}\label{not so fast, buster}
Suppose that $T$ has at most $l\geq 2$ leaves
and $(R,W,\underline{U})\in F(T)$.
Then $$|R| + |R_W| \leq 2|W| + l - 2.$$
\end{corollary}

\begin{proof}
First, we note that since $T$ has at most $l$ leaves and $\underline{U}\in\Comp_R(T)$, $R$ must also have at most $l$ leaves.
Let $\overline{R}$ be the tree obtained from $R$ by contracting the edges of $R_W$, and let $\overline{W}$
be the image of $W$ in the vertex set of $\overline{R}$.  Then $\overline{R}$ is a tree with at most $l$ leaves,
and every edge in $\overline{R}$ has exactly one vertex in $\overline{W}$.
Furthermore, we have $|R| + |R_W| - 2|W| = |\overline{R}| - 2|\overline{W}|$, so we need to prove that
$|\overline{R}| \leq 2|\overline{W}| + l - 2$.  This follows from Lemma \ref{leaves}.
\end{proof}

\subsection{Kazhdan-Lusztig coefficients of cones over trees}\label{sec:kl-thm}
Let $\IH_{2i}^\cone$ be the $\cTop$-module that assigns to any tree $T$ the vector space $\IH_{2i}\big(X_{\cone(T)}\big)$,
which is well defined by Remark \ref{simplification}.
We are now ready to state and prove Theorem \ref{IH-small}.

\begin{proof}[Proof of Theorem \ref{IH-small}]
Suppose that $R$ is a tree and $W\subset \Vert(R)$ is groovy.
Fix a pair of natural numbers $p,q\leq 2i$, and for each $v\in \Vert(R)$, define the $\cTop$-module
$$N_{v} := \begin{cases}\OS_* & \text{if $v\in W$} \\ \OS_*^\cone & \text{if $v\notin W$.}\end{cases}$$
Consider the $\cTop$-module $$N_{(R,W)}(T) \;\; := \bigoplus_{\underline U\in \Comp_R(T)} \;
\left(\bigotimes_{v\in \Vert(R)} N_{v}(U_v)\right)_{\!\! 2i-p-q}\bigotimes\;\;\; \IH_{2(i-q)}\big(\cone(R_W)\big).$$
Then Theorem \ref{spectral} gives us a spectral sequence $E(\cone(-),i)$ in the category of $\cTop$-modules,
converging to $\IH_{2i}^\cone$, with
$$E(\cone(-),i)_{p,q}^{1} \;\;\;\cong \;\;\;
\bigoplus_R \left(
\bigoplus_{\substack{W\subset \Vert(R)\;\text{groovy}\\ |W|=p}}
N_{(R,W)}\right)^{\operatorname{Aut}(R)}.$$
Here the outer direct sum is over isomorphism classes of trees, and the superscript denotes invariants under the action of the group of automorphisms of $R$.
By taking invariants, we ensure that for each $T$ we obtain a sum over flats of $\cone(T)$ 
rather than over elements of $F(T)$.

The direct sum in our expression for $E(\cone(-),i)_{p,q}^{1}$ is not finite.  However, if we restrict to the subcategory $\cTop_l$, Corollary \ref{not so fast, buster}
tells us that $N_{(R,W)}$ vanishes unless $|R| + |R_W| \leq 2p + l - 2$.  Since there are finitely many trees
with at most $2p+l-2$ edges, the direct sum becomes finite.

Note that, when we say that the spectral sequence converges to $\IH_{2i}^\cone$, the precise meaning of this statement
is that the module $\IH_{2i}^\cone$ admits a filtration whose associated graded is isomorphic to the infinity page of the spectral sequence.
In particular, if we can prove that $N_{(R,W)}$ is $(2i+l-2)$-small, this will imply that the first page of the spectral sequence
is $(2i+l-2)$-small, and therefore that the infinity page is also $(2i+l-2)$-small, and finally that
$\IH_{2i}^\cone$ is $(2i+l-2)$-smallish.
By Proposition \ref{OSi} and Lemma \ref{putting them together}, $N_{(R,W)}$ is $\big(|R| + 2i-p-q\big)$-small.
We will complete the proof by showing that $\IH_{2(i-q)}\big(\cone(R_W)\big) = 0$ unless $|R| + 2i-p-q\leq2i+l-2$.  

We may write $R_W \cong R_1 \sqcup \cdots \sqcup R_k$ as a disjoint union of trees.  
We have $$X_{\cone(R_W)} \cong X_{\cone(R_1)}\times\cdots\times X_{\cone(R_k)},$$ thus
the K\"unneth theorem tells us that
\begin{eqnarray*}\IH_{2(i-q)}\big(X_{\cone(R_W)}\big) &\cong& 
\Big(\IH_*\big(X_{\cone(R_1)}\big)\otimes\cdots\otimes\IH_*\big(X_{\cone(R_k)}\big)\Big)_{\! 2i}\\
&\cong& \bigoplus_{r_1 + \cdots + r_k = i-q} 
\IH_{2r_1}\big(X_{\cone(R_1)}\big)\otimes\cdots\otimes \IH_{2r_k}\big(X_{\cone(R_k)}\big).
\end{eqnarray*}
We also have $\IH_{2r_j}\big(X_{\cone(R_j)}\big) = 0$ unless $2r_j \leq |R_j|$ \cite[Proposition 3.4]{EPW}, thus 
this direct sum vanishes
unless $2(i-q) \leq \sum |R_j| = |R_W|$.
This means that, for $\IH_{2(i-q)}\big(X_{\cone(R_W)}\big)$ to be nonzero, we must have
$$|R|+2i-p-q = |R| + 2(i-q) - p +q \leq |R| + |R_W| - p + q.$$
By Corollary \ref{not so fast, buster}, this is at most $p+q + l -2$,
which is in turn bounded above by $2i + l -2$.
\end{proof}

\subsection{Examples}\label{kl-examples}
We end with four families of examples to illustrate Theorem \ref{IH-small}.

\begin{example}\label{paths}
Let $I_m$ be the path of length $m$.
Theorem \ref{IH-small} says that the restriction of $\IH_{2i}^\cone$ to the opposite category of paths is $2i$-smallish, 
and then Theorem \ref{actual polynomial}
says that the dimension of $\IH_{2i}\big(\cone(I_m)\big)$ is eventually a polynomial in $m$ of degree at most $2i$.
The cone on a path is a fan (see the picture in Example \ref{cone-flat}), so \cite[Theorem 1.1]{fan-wheel-whirl}
gives the precise formula $$\dim \IH_{2i}\big(\cone(I_m)\big) = \frac{1}{i+1}\binom{m}{i,i,m-2i} = \frac{1}{i!(i+1)!}\, m(m-1)\cdots(m-2i+1).$$
This is indeed a polynomial in $m$ of degree $2i$, which means that our smallishness result is in fact the best possible. 
\end{example}

The next three examples will use the fact that, 
for any graph $G$, $\dim \IH_2(X_G)$ is equal to the number of corank 1 flats of $G$
minus the number of rank 1 flats of $G$ \cite[Proposition 2.12]{EPW}.
Corank 1 flats of the cone over a tree are in bijection with subtrees 
(the corank 1 flat associated with a triple $(R,W,\underline{U})$ with $|W|=1$
corresponds the subtree $U_v$ for the unique element $v\in W$), 
while rank 1 flats are edges, which are in bijection with edges and vertices of the original tree.

\begin{example}
Let us consider the restriction of $\IH_2^\cone$ to the category $\cTop_3$.  
Let $T = K_{3,1}$ be the tree with edges $e_1$, $e_2$, and $e_3$ meeting at a single vertex, and let 
$\underline{e} = (e_1,e_2,e_3)$.  Every object of $\cTop_3$ is isomorphic to 
$T_{\underline{e}}(\underline{m})$ for some 3-tuple $\underline{m}$ of natural numbers.
Theorem \ref{IH-small} says that our functor is $3$-smallish, 
and then Proposition \ref{actual polynomial}
says that the dimension of $\IH_{2}^\cone(T_{\underline{e}}(\underline{m}))$ is eventually a polynomial of degree at most 3 
in the variables $m_1,m_2,m_3$.

The number of subtrees of $T_{\underline{e}}(\underline{m})$ is equal to 
$$(m_1+1)(m_2+1)(m_3+1) + \binom{m_1+1}{2} + \binom{m_2+1}{2} + \binom{m_2+1}{2},$$
where the first term counts subtrees that contain the vertex of degree 3, while the next three terms count subtrees that touch
only one of the three tails.  The number of edges of $\cone(T_{\underline{e}}(\underline{m}))$ 
is equal to $2(m_1+m_2+m_3)+1$.
We therefore have \begin{eqnarray*}\dim \IH_2^\cone(T_{\underline{e}}(\underline{m})) 
&=&  (m_1+1)(m_2+1)(m_3+1) + \binom{m_1+1}{2} + \binom{m_2+1}{2} + \binom{m_2+1}{2}\\
&& - 2(m_1+m_2+m_3) - 1,\end{eqnarray*} which is indeed a polynomial of degree 3.  
Thus our result that $\IH_2$ is 3-smallish is again the best possible.
\end{example}

\begin{example}
Let $T$ be an arbitrary $T$ and $e$ an edge of $T$.  Let $T_e(m)$ be the tree obtained by subdividing $e$ into $m$ edges.
(In other words, we take $r=1$ and drop the underlines from the notation.)
Theorem \ref{IH-small} and Proposition \ref{actual polynomial} combine to say that $\dim \IH_2^\cone(T_{e}(m))$ is eventually
polynomial in $m$ of degree at most $l$, where $l$ is the number of leaves of $T$.

The number of edges of $\cone(T_e(m))$ is equal to $|T| + m - 1$, which is linear in $m$.
There are three types of subtrees of $T_e(m)$:  those that are disjoint from the set of subdivided edges, those that are contained in the
set of subdivided edges, and all the rest.  The number of subtrees that are disjoint from the set of subdivided edges 
is independent of $m$, the number of subtrees that are contained in the set of subdivided elements is equal to $\binom{m+1}{2}$,
and the number of remaining subtrees is linear in $m$.  Thus $\dim \IH_2^\cone(T_{e}(m))$ is equal to $\binom{m+1}{2} + O(m)$.
Our result on the growth of this dimension is therefore sharp if and only if $l=2$ (Example \ref{paths}).
\end{example}

\begin{example}\label{ex:thag}
Consider the tree $K_{m,1}$ with one central vertex connected to $m$ satellites.
The number of subtrees of $K_{m,1}$ is equal to $2^m + m$,
and the number of edges of $\cone(K_{m,1})$ is $2m+1$,
so we have
$\dim\IH_2^\cone(K_{m,1}) = 2^m - m - 1$.
This is clearly not bounded above by a polynomial in $|K_{m,1}| = m$, which reflects the fact that there is no subcategory 
$\cT_l\subset \cT$ that contains every $K_{m,1}$ and proves that $\IH_2^\cone$ is not finitely generated
as a $\cTop$-module.
The cone over $K_{m,1}$ is a {\bf Thagomizer graph}, and
the dimension of $\IH_{2i}^\cone(K_{m,1})$ for arbitrary $i$ and $m$ is computed in \cite[Theorem 1.1]{thag}.
\end{example}

\bibliography{./symplectic}
\bibliographystyle{amsalpha}

\end{document}